\documentclass[12pt]{amsart}
\usepackage[english,  activeacute]{babel}
\usepackage[latin1]{inputenc}
\usepackage{amsmath}
\usepackage{amssymb}
\usepackage{amsthm}
\usepackage{MnSymbol}
\usepackage{pdflscape}
\usepackage{a4wide}
\usepackage{url}
\usepackage{gastex}
\usepackage{pst-3dplot}
\theoremstyle{plain}
\usepackage{float}

\newtheorem{theorem}{Theorem}
\newtheorem{corollary}[theorem]{Corollary}
\newtheorem{lemma}[theorem]{Lemma}
\newtheorem{proposition}[theorem]{Proposition}
\theoremstyle{definition}
\newtheorem{definition}[theorem]{Definition}
\newtheorem{example}[theorem]{Example}

\begin{document}

\title[Enumeration of  $k$-Fibonacci Paths]{Enumeration of  $k$-Fibonacci Paths using Infinite Weighted Automata}

\author{Rodrigo De Castro}
\address{Universidad Nacional de Colombia \\ Departamento de Matem\'aticas \\ AA 14490\\
 Bogot\'a\\ Colombia}
\email{rdcastrok@unal.edu.co}

\author{Jos\'e L. Ram\'irez}
\address{ Universidad Nacional de Colombia \\ Departamento de Matem\'aticas \\ AA 14490 \\ and \\ Universidad Sergio Arboleda \\ Instituto de Matem\'aticas y sus Aplicaciones \\ Bogot\'a\\ Colombia }
\email{jlramirezr@unal.edu.co;  josel.ramirez@ima.usegioarboleda.edu.co}
%\thanks{The second  author was partially supported by Universidad Sergio Arboleda under Grant no. USA-II-2011-0059.}

\begin{abstract}
In this paper, we introduce a new family of generalized colored Motzkin paths, where horizontal steps are colored by means of $F_{k,l}$ colors, where $F_{k,l}$ is the $l$-th $k$-Fibonacci number. We study the enumeration of this family according to  the length. For this, we use infinite weighted automata.
\end{abstract}



\maketitle

\section{Introduction}
A lattice path of length $n$ is a sequence of points $P_1, P_2, \dots, P_n$ with $n\geqslant1$ such that each point $P_i$ belongs to the plane integer lattice and each two consecutive points $P_i$ and $P_{i+1}$ connect by a line segment.  We will consider lattice paths in  $\mathbb{Z}\times\mathbb{Z}$  using three step types: a rise step $U=(1,1)$, a fall step $D=(1,-1)$ and a $F_{k,l}-$colored length horizontal step $H_l=(l,0)$ for every positive integer $l$, such that $H_l$ is colored by means of $F_{k,l}$ colors, where $F_{k,l}$ is the $l$-th $k$-Fibonacci number.

Many kinds of generalizations of the Fibonacci Numbers have been presented in the literature  \cite{koshy, Hor} and the corresponding references. One of them is the \emph{$k$-Fibonacci Numbers}. For any positive integer number $k$, the $k$-Fibonacci sequence, say $\{F_{k,n}\}_{n\in \mathbb{N}}$, is defined recurrently by
\begin{align*}
F_{k,0}=0, \  \ F_{k,1}=1,  \   \  F_{k,n+1}=kF_{k,n}+F_{k,n-1},  \ \text{for} \ n\geqslant 1. %\label{eq1}
\end{align*}
The  generating function of the $k$-Fibonacci numbers is  $f_k(x)=\frac{x}{1-kx-x^2}$, \cite{Falcon2, Falcon1}. This sequence was studied by Horadam in \cite{Horadam}. Recently, Falc\'on and Plaza \cite{Falcon1} found the $k$-Fibonacci numbers by studying the recursive application of two geometrical transformations used in the  four-triangle longest-edge (4TLE) partition. The interested reader is also referred to \cite{CEN, Falcon4, Falcon2, Falcon5, Falcon1,   RAM2, RAM, Asalas}, for further information about this.

A \emph{generalized $F_{k,l}$-colored Motzkin path} or simply \emph{$k$-Fibonacci path} is a sequence of rise, fall and $F_{k,l}-$colored length horizontal steps $(l=1,2,\dots)$ running from $(0,0)$ to $(n,0)$ that never pass below the $x$-axis. We denote by $\mathcal{M}_{F_{k,n}}$ the set of all $k$-Fibonacci paths of length $n$ and $\mathcal{M}_k=\bigcup_{n=0}^{\infty}\mathcal{M}_{F_{k,n}}$.  In Figure \ref{fig1} we show the set $\mathcal{M}_{F_{2,3}}$.

   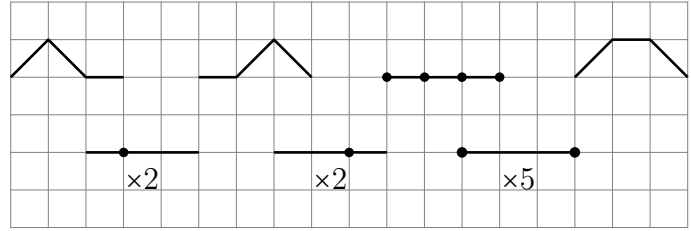
\begin{figure}[h]
\centering
\psset{unit=5mm}
\begin{pspicture}(0,-4)(18,2)
\psgrid[gridwidth=0.3pt,gridcolor=gray,subgriddiv=0, gridlabels=0]
\psline[linewidth=1pt]{-}(0,0)(1,1)(2,0)(3,0)
\psline[linewidth=1pt]{-}(5,0)(6,0)(7,1)(8,0)
\psline[linewidth=1pt](10,0)(13,0)
\psdots(10,0)(11,0)(12,0)(13,0)
\psline[linewidth=1pt]{-}(15,0)(16,1)(17,1)(18,0)
\psline[linewidth=1pt]{-}(2,-2)(5,-2)
\psdots(3,-2)(9,-2)
\psline[linewidth=1pt]{-}(7,-2)(10,-2)
\psline[linewidth=1pt]{*-*}(12,-2)(15,-2)
\rput(3.5,-2.7){$\times 2$}
\rput(8.5,-2.7){$\times 2$}
\rput(13.5,-2.7){$\times 5$}
\end{pspicture}
\caption{$k$-Fibonacci Paths of length 3, $|\mathcal{M}_{F_{2,3}}|=13$.}
 \label{fig1}
\end{figure}

A \emph{grand  $k$-Fibonacci path} is a $k$-Fibonacci path without the condition that never going below the $x$-axis. We denote by $\mathcal{M}_{F_{k,n}}^*$ the set of all grand $k$-Fibonacci paths of length $n$ and $\mathcal{M}_k^*=\bigcup_{n=0}^{\infty}\mathcal{M}_{F_{k,n}}^*$.  A \emph{prefix $k$-Fibonacci path} is a  $k$-Fibonacci path without the condition that ending on the $x$-axis. We denote by $\mathcal{PM}_{F_{k,n}}$ the set of all prefix $k$-Fibonacci paths of length $n$ and $\mathcal{PM}_k=\bigcup_{n=0}^{\infty}\mathcal{PM}_{F_{k,n}}$. Analogously, we have the family of \emph{prefix grand $k$-Fibonacci paths}.  We denote by $\mathcal{PM}_{F_{k,n}}^*$ the set of all prefix grand $k$-Fibonacci paths of length $n$ and $\mathcal{PM}_k^*=\bigcup_{n=0}^{\infty}\mathcal{PM}_{F_{k,n}}^*$.

In this paper, we study the generating function for the $k$-Fibonacci paths, grand $k$-Fibonacci paths, prefix $k$-Fibonacci paths, and prefix grand $k$-Fibonacci paths, according to the length. We use Counting Automata Methodology (CAM) \cite{ROD}, which is a variation of the methodology developed by Rutten \cite{RUT} called Coinductive Counting.  Counting Automata Methodology uses infinite weighted automata, weighted graphs and continued fractions.  The main idea of this methodology is find a counting automaton such that there exist a bijection between all words recognized by an automaton $\mathcal{M}$  and the family of  combinatorial objects.  From the counting automaton $\mathcal{M}$ is possible find the  ordinary generating function (GF)  of the family of combinatorial objects \cite{ROD}.

\section{Counting Automata Methodology}
The terminology and notation are mainly those of Sakarovitch \cite{JAC}.  An  \emph{automaton} $\mathcal{M}$ is a 5-tuple $\mathcal{M}=\left(\Sigma, Q, q_{0}, F, E\right)$, where $\Sigma$  is a nonempty input alphabet, $Q$ is a nonempty set of states of $\mathcal{M}$, $q_{0}\in Q$ is the initial state of $\mathcal{M}$,  $\emptyset \neq F\subseteq Q$ is the set of final states of $\mathcal{M}$ and $E \subseteq Q \times \Sigma \times Q$ is the set of transitions of $\mathcal{M}$.  The language recognized  by an automaton   $\mathcal{M}$ is denoted by  $L(\mathcal{M})$. If $Q, \Sigma$ and $E$ are finite sets, we say that  $\mathcal{M}$ is a finite automaton \cite{JAC}.

\begin{example}\label{eje1}
Consider the finite automaton $\mathcal{M}=\left(\Sigma, Q, q_{0}, F, E\right)$ where $\Sigma=\left\{a, b\right\}$, $Q=\left\{q_{0}, q_{1}\right\}$, $F=\left\{q_{0}\right\}$ and $E=\{(q_0,a,q_1), (q_0,b,q_0), (q_1,a,q_0) \}$.
The transition diagram of $\mathcal{M}$ is as shown in Figure \ref{ejediagrama}. It is easy to verify that  $L(\mathcal{M})=(b \cup aa)^*$.
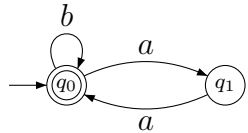
\begin{figure}[h]
  \begin{center}
    \unitlength=3pt
    \begin{picture}(20,10)(0,-1)
    %\put(0,-30){\framebox(30,30){}}
    \gasset{Nw=5,Nh=5,Nmr=2.5,curvedepth=0}
    \thinlines
    \node[Nmarks=ir,iangle=180](A0)(0,0){\tiny{$q_{0}$}}
    \node(A1)(20,0){\tiny{$q_{1}$}}
    \drawloop[loopdiam=4,loopangle=90](A0){$b$}
    \drawedge[curvedepth=3](A0,A1){$a$}
    \drawedge[curvedepth=3](A1,A0){$a$}
    \end{picture}
  \end{center}
  \caption{Transition diagram of $\mathcal{M}$, Example  \ref{eje1}.} \label{ejediagrama}
\end{figure}

 \end{example}

\begin{example}\label{eje2}
Consider the infinite automaton $\mathcal{M_D}=\left(\Sigma, Q, q_{0}, F, E\right)$, where $\Sigma=\left\{a, b\right\}$, $Q=\left\{q_{0}, q_{1}, \ldots\right\}$, $F=\left\{q_{0}\right\}$ and $E=\left\{(q_i, a, q_{i+1}), (q_{i+1}, b, q_{i}): i\in \mathbb{N}\right\}$.
The transition diagram of $\mathcal{M_D}$ is as shown in Figure \ref{ejediagrama2}.

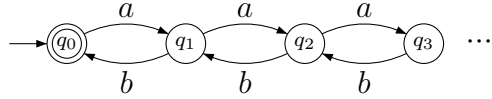
\begin{figure}[h]
 \centering
    \unitlength=3pt
    \begin{picture}(52, 6)(0,-1)
    %\put(0,0){\framebox(52,5){}}
        \gasset{Nw=5,Nh=5,Nmr=2.5,curvedepth=0}
    \thinlines
    \node[Nmarks=ir,iangle=180, curvedepth=3](A0)(0,1){\tiny{$q_{0}$}}
    \node(A1)(15,1){\tiny{$q_{1}$}}
    \node(A2)(30,1){\tiny{$q_{2}$}}
    \node(A3)(45,1){\tiny{$q_{3}$}}
    \drawedge[curvedepth=2.5](A0,A1){$a$}
    \drawedge[curvedepth=2.5](A1,A0){$b$}
    \drawedge[curvedepth=2.5](A1,A2){$a$}
    \drawedge[curvedepth=2.5](A2,A1){$b$}
    \drawedge[curvedepth=2.5](A2,A3){$a$}
    \drawedge[curvedepth=2.5](A3,A2){$b$}
    \gasset{Nframe=n,Nadjust=w,Nh=6,Nmr=0}
    \node(P)(52,1){$\cdots$}
    \end{picture}
  \caption{Transition  diagram of  $\mathcal{M_D}$, Example \ref{eje2}.}
  \label{ejediagrama2}
\end{figure}

The language accepted by $\mathcal{M_D}$ is $$L(\mathcal{M_D})=\left\{w\in\Sigma^*: |w|_a=|w|_b \ \text{and for all prefix $v$ of} \ w, |v|_b\leq|v|_a \right\}.$$ \end{example}

An ordinary generating function  $F=\sum_{n=0}^{\infty}f_nz^n$ corresponds to a formal language $L$ if $f_n=\left|\left\{w\in L: \left|w\right|=n\right\}\right|$, i.e., if the $n$-th coefficient $f_n$ gives the number of words in $L$ with length $n$.

Given an alphabet $\Sigma$ and a semiring $\mathbb{K}$. A \emph{formal power series} or \emph{formal series} $S$ is a function $S:\Sigma^*\rightarrow \mathbb{K}$. The image of a word $w$ under $S$ is called the \emph{coefficient} of $w$ in $S$ and is denoted by $s_w$. The series $S$ is written as a formal sum $S = \sum_{w\in \Sigma^*} s_ww$. The set of formal power series over $\Sigma$ with coefficients in $\mathbb{K}$ is denoted by $\mathbb{K}\left\langle\langle \Sigma^* \right\rangle\rangle$.\\

An automaton over $\Sigma^*$  with weights in $\mathbb{K}$, or \emph{$\mathbb{K}$-automaton} over $\Sigma^*$ is a graph labelled with elements of $\mathbb{K}\left\langle\langle \Sigma^* \right\rangle\rangle$, associated with two maps from the set of vertices to $\mathbb{K}\left\langle\langle \Sigma^* \right\rangle\rangle$. Specifically, a \emph{weighted automaton} $\mathcal{M}$ over $\Sigma^*$ with weights in $\mathbb{K}$ is a 4-tuple  $\mathcal{M}=\left(Q, I, E, F\right)$ where $Q$ is a nonempty set of \emph{states} of $\mathcal{M}$, $E$ is an element of  $\mathbb{K}\left\langle\langle \Sigma^* \right\rangle\rangle^{Q\times Q}$ called \emph{transition matrix}. $I$ is an element of $\mathbb{K}\left\langle\langle \Sigma^* \right\rangle\rangle^{Q}$, i.e., $I$ is a function from $Q$ to $\mathbb{K}\left\langle\langle \Sigma^* \right\rangle\rangle$. $I$  is the \emph{initial function} of  $\mathcal{M}$ and can also be seen as a row vector of dimension $Q$, called \emph{initial vector} of  $\mathcal{M}$ and $F$ is an element of $\mathbb{K}\left\langle\langle \Sigma^* \right\rangle\rangle^{Q}$. $F$  is the \emph{final function} of $\mathcal{M}$ and  can also be seen as a column vector of dimension $Q$, called \emph{final vector} of  $\mathcal{M}$.

We say that $\mathcal{M}$ is a \emph{counting automaton} if $\mathbb{K}=\mathbb{Z}$ and $\Sigma^*=\left\{z\right\}^*$. With each automaton, we can associate a counting automaton. It can be obtained from a given automaton replacing every transition labelled with a symbol $a$, $a\in\Sigma$, by a transition labelled with $z$. This transition is called a \emph{counting transition} and the  graph is called  a \emph{counting automaton} of  $\mathcal{M}$. Each transition from $p$ to $q$ yields an equation
\begin{align*}
L(p)(z)=zL(q)(z) + \left[p\in F\right] + \cdots.
\end{align*}
We use  $L_p$ to denote  $L(p)(z)$. We also use Iverson's notation, $\left[P\right] =1$ if the proposition $P$ is true and  $\left[P\right] =0$ if $P$ is false.

\subsection{Convergent Automata and Convergent Theorems}
We denote by $L^{(n)}(\mathcal{M})$ the number of words of length $n$ recognized by the automaton $\mathcal{M}$, including repetitions.
\begin{definition}
We say that an automaton $\mathcal{M}$ is convergent if for all integer $n\geqslant 0$, $L^{(n)}(\mathcal{M})$ is finite.
\end{definition}

The proof of following theorems and propositions can be found in \cite{ROD}.

\begin{theorem}[First Convergence Theorem]\label{teorema1conv}
Let $\mathcal{M}$ be an automaton such that each vertex (state) of the counting automaton of  $\mathcal{M}$ has finite degree. Then $\mathcal{M}$ is convergent.
\end{theorem}

\begin{example}
The counting automaton of the automaton  $\mathcal{M_D}$ in Example \ref{eje2} is convergent.
\end{example}
The following definition plays an important role in the development of applications because it allows to simplify counting automata whose transitions are formal series.

\begin{definition}
Let $\mathcal{M}$ be an automaton, and let $f(z)=\sum_{n=0}^{\infty}f_nz^n $  be a formal power series with  $f_{n} \in \mathbb{N}$ for all $n\geqslant 0$ and $f_{0}=0$. In a counting automaton of $\mathcal{M}$ the set of counting transitions  from state $p$ to state $q$, without intermediate final states, see Figure \ref{paralelo1} (left), is represented by a graph with a single edge labeled by $f(z)$, see  Figure \ref{paralelo1} (right).

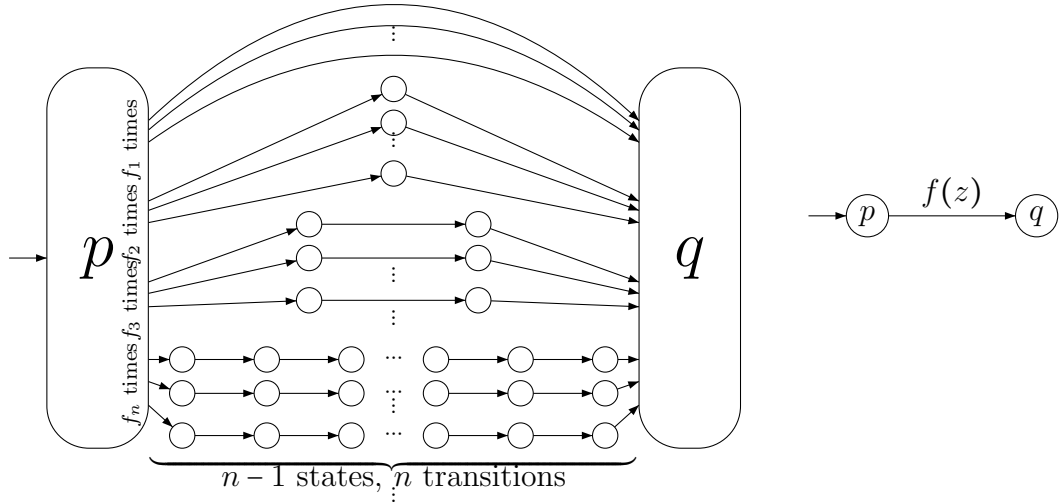
\begin{figure}[h]
\centering
   \unitlength=3.2pt
    \begin{picture}(106,58)(-5,3)
    %\put(0,0){\framebox(60,60){}}
      \gasset{Nw=5,Nh=5,Nmr=2.5,curvedepth=0}
    \thinlines
    \node[Nmarks=i, iangle=180, Nw=12, Nadjustdist=5, Nh=45, Nmr=5](A0)(-5,30){\Huge{$p$}}
    \node[Nw=12, Nadjustdist=5, Nh=45, Nmr=5](A1)(65,30){\Huge $q$}
    \drawedge[curvedepth=20, syo=10,eyo=10](A0,A1){}
    \drawedge[curvedepth=18, syo=9.5,eyo=9.5](A0,A1){}
    \drawedge[curvedepth=15, syo=9,eyo=9](A0,A1){}
    \node[Nh=3,Nw=3,Nmr=1.5](A2)(30,50){}
     \node[Nh=3,Nw=3,Nmr=1.5](A3)(30,46){}
     \node[Nh=3,Nw=3,Nmr=1.5](A4)(30,40){}
    \drawedge[syo=4](A0,A2){}
    \drawedge[syo=3.5](A0,A3){}
    \drawedge[syo=3](A0,A4){}
    \drawedge[eyo=4](A2,A1){}
    \drawedge[eyo=3.5](A3,A1){}
    \drawedge[eyo=3](A4,A1){}
      \node[Nh=3,Nw=3,Nmr=1.5](A5)(20,34){}
     \node[Nh=3,Nw=3,Nmr=1.5](A55)(40,34){}
      \node[Nh=3,Nw=3,Nmr=1.5](A6)(20,30){}
     \node[Nh=3,Nw=3,Nmr=1.5](A66)(40,30){}
      \node[Nh=3,Nw=3,Nmr=1.5](A7)(20,25){}
     \node[Nh=3,Nw=3,Nmr=1.5](A77)(40,25){}
     \drawedge[syo=-5](A0,A5){}
     \drawedge(A5,A55){}
      \drawedge[eyo=-5](A55,A1){}
    \drawedge[syo=-5.5](A0,A6){}
     \drawedge(A6,A66){}
      \drawedge[eyo=-5.5](A66,A1){}
    \drawedge[syo=-6](A0,A7){}
     \drawedge(A7,A77){}
      \drawedge[eyo=-6](A77,A1){}
       \node[Nh=3,Nw=3,Nmr=1.5](A8)(5,18){}
     \node[Nh=3,Nw=3,Nmr=1.5](A88)(15,18){}
      \node[Nh=3,Nw=3,Nmr=1.5](A888)(25,18){}
     \node[Nh=3,Nw=3,Nmr=1.5](A82)(35,18){}
      \node[Nh=3,Nw=3,Nmr=1.5](A822)(45,18){}
     \node[Nh=3,Nw=3,Nmr=1.5](A8222)(55,18){}
        \node[Nh=3,Nw=3,Nmr=1.5](A9)(5,14){}
     \node[Nh=3,Nw=3,Nmr=1.5](A99)(15,14){}
      \node[Nh=3,Nw=3,Nmr=1.5](A999)(25,14){}
     \node[Nh=3,Nw=3,Nmr=1.5](A92)(35,14){}
      \node[Nh=3,Nw=3,Nmr=1.5](A922)(45,14){}
     \node[Nh=3,Nw=3,Nmr=1.5](A9222)(55,14){}
      \node[Nh=3,Nw=3,Nmr=1.5](A10)(5,9){}
     \node[Nh=3,Nw=3,Nmr=1.5](A110)(15,9){}
      \node[Nh=3,Nw=3,Nmr=1.5](A1110)(25,9){}
     \node[Nh=3,Nw=3,Nmr=1.5](A12)(35,9){}
      \node[Nh=3,Nw=3,Nmr=1.5](A112)(45,9){}
     \node[Nh=3,Nw=3,Nmr=1.5](A1112)(55,9){}
     \drawedge[syo=-12](A0,A8){}
     \drawedge(A8,A88){}
      \drawedge(A88,A888){}
    \drawedge(A82,A822){}
     \drawedge(A822,A8222){}
      \drawedge[eyo=-12](A8222,A1){}
    \drawedge[syo=-12.5](A0,A9){}
     \drawedge(A9,A99){}
      \drawedge(A99,A999){}
    \drawedge(A92,A922){}
     \drawedge(A922,A9222){}
      \drawedge[eyo=-12.5](A9222,A1){}
    \drawedge[syo=-12](A0,A10){}
     \drawedge(A10,A110){}
      \drawedge(A110,A1110){}
    \drawedge(A12,A112){}
     \drawedge(A112,A1112){}
      \drawedge[eyo=-12](A1112,A1){}
      \gasset{Nw=5,Nh=5,Nmr=2.5,curvedepth=0}
    \thinlines
    \node[Nmarks=i,iangle=180](A0)(86,35){$p$}
    \node(A1)(106,35){$q$}
    \drawedge(A0,A1){$f(z)$}

     \gasset{Nframe=n,Nadjust=w,Nh=6,Nmr=0}
    \node(P)(30,56.5){\footnotesize{\tiny{$\vdots$}}}
    \node(P)(30,44){\footnotesize{\tiny{$\vdots$}}}
    \node(P)(30,28){\footnotesize{\tiny{$\vdots$}}}
   \node(P)(30,12.5){\tiny{$\vdots$}}
    \node(P)(30,23){\tiny{$\vdots$}}
   \node(P)(30,18){\tiny{$\cdots$}}
   \node(P)(30,14){\tiny{$\cdots$}}
   \node(P)(30,9){\tiny{$\cdots$}}
   \node(P)(-1,44){\rotatebox{90}{\tiny $f_1$  times}}
   \node(P)(-1,34){\rotatebox{90}{\tiny$f_2$ times}}
   \node(P)(-1,25){\rotatebox{90}{\tiny$f_3$ times}}
   \node(P)(-1,15){\rotatebox{90}{\tiny $f_n$ times}}
   \node(P)(30,6){$\underbrace{ \hspace{6.5cm}  }$}
   \node(P)(30,4){$n-1$ states, $n$ transitions}
   \node(P)(30,2){\tiny{$\vdots$}}
         \end{picture}
      \caption{Transitions  from the state $p$ to $q$ and its transition in parallel.}
         \label{paralelo1}
         \end{figure}
This kind of transition is called a transition in parallel. The states $p$ and $q$  are called  visible states and the intermediate states are called  hidden states.
\end{definition}
\begin{example}\label{ejeconteoo}
In Figure  \ref{ejeautconteo} (left)  we display a counting  automaton $\mathcal{M}_1$ without  transitions in parallel, i.e., every transition is label by $z$. The transitions from state  $q_1$ to  $q_2$ correspond to the series  $\frac{1-\sqrt{1-4z}}{2}=z+z^2+2z^3+5z^4+14z^5+\cdots$. However, this automaton  can also be represented using transitions in parallel. Figure  \ref{ejeautconteo} (right) displays two examples.
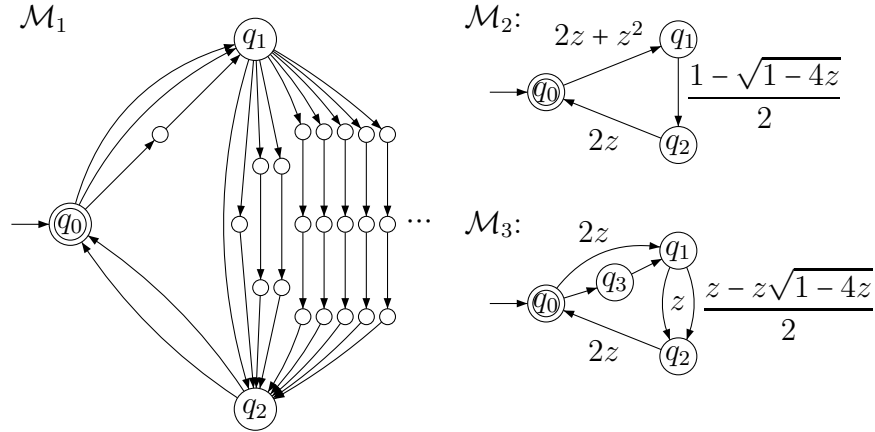
\begin{figure}[h]
\centering
   \unitlength=2pt
    \begin{picture}(130,78)(0,0)
    %\put(0,0){\framebox(35,40){}}
      \gasset{Nw=5,Nh=5,Nmr=2.5,curvedepth=0}
    \thinlines
    \node[Nmarks=ri, iangle=180, Nw=8, Nadjustdist=8, Nh=8, Nmr=8](A0)(0,35){$q_0$}
    \node[Nw=8, Nadjustdist=8, Nh=8, Nmr=8](A1)(35,70){$q_1$}
   \node[Nw=8, Nadjustdist=8, Nh=8, Nmr=8](A2)(35,0){$q_2$}
   \node[Nw=3, Nadjustdist=3, Nh=3, Nmr=3](A011)(17,52){}
   \node[Nw=3, Nadjustdist=3, Nh=3, Nmr=3](A121)(32,35){}
   \node[Nw=3, Nadjustdist=3, Nh=3, Nmr=3](A122)(36,23){}
   \node[Nw=3, Nadjustdist=3, Nh=3, Nmr=3](A123)(36,46){}
   \node[Nw=3, Nadjustdist=3, Nh=3, Nmr=3](A1233)(40,23){}
   \node[Nw=3, Nadjustdist=3, Nh=3, Nmr=3](A1234)(40,46){}
   \node[Nw=3, Nadjustdist=3, Nh=3, Nmr=3](A124)(44,17.5){}
   \node[Nw=3, Nadjustdist=3, Nh=3, Nmr=3](A125)(44,35){}
   \node[Nw=3, Nadjustdist=3, Nh=3, Nmr=3](A126)(44,52.5){}
   \node[Nw=3, Nadjustdist=3, Nh=3, Nmr=3](A127)(48,17.5){}
   \node[Nw=3, Nadjustdist=3, Nh=3, Nmr=3](A128)(48,35){}
   \node[Nw=3, Nadjustdist=3, Nh=3, Nmr=3](A129)(48,52.5){}
   \node[Nw=3, Nadjustdist=3, Nh=3, Nmr=3](A1210)(52,17.5){}
   \node[Nw=3, Nadjustdist=3, Nh=3, Nmr=3](A1211)(52,35){}
   \node[Nw=3, Nadjustdist=3, Nh=3, Nmr=3](A1212)(52,52.5){}
   \node[Nw=3, Nadjustdist=3, Nh=3, Nmr=3](A1213)(56,17.5){}
   \node[Nw=3, Nadjustdist=3, Nh=3, Nmr=3](A1214)(56,35){}
   \node[Nw=3, Nadjustdist=3, Nh=3, Nmr=3](A1215)(56,52){}
   \node[Nw=3, Nadjustdist=3, Nh=3, Nmr=3](A1216)(60,17.5){}
      \node[Nw=3, Nadjustdist=3, Nh=3, Nmr=3](A1217)(60,35){}
         \node[Nw=3, Nadjustdist=3, Nh=3, Nmr=3](A1218)(60,52){}
   \drawedge[curvedepth=5](A0,A1){}
   \drawedge[curvedepth=8](A0,A1){}
   \drawedge(A0,A011){}
   \drawedge(A011,A1){}
   \drawedge[curvedepth=-6](A1,A2){}
   \drawedge[curvedepth=-3](A2,A0){}
  \drawedge[curvedepth=3](A2,A0){}
  \drawedge(A1,A121){}
  \drawedge(A121,A2){}
  \drawedge(A1,A1234){}
  \drawedge(A1234,A1233){}
  \drawedge(A1233,A2){}
    \drawedge(A1,A123){}
  \drawedge(A123,A122){}
  \drawedge(A122,A2){}
  \drawedge[curvedepth=1](A1,A1215){}
  \drawedge(A1215,A1214){}
  \drawedge(A1214,A1213){}
  \drawedge[curvedepth=1](A1213,A2){}
  \drawedge[curvedepth=1](A1,A1212){}
  \drawedge(A1212,A1211){}
  \drawedge(A1211,A1210){}
  \drawedge[curvedepth=1](A1210,A2){}
  \drawedge[curvedepth=1](A1,A129){}
  \drawedge(A129,A128){}
  \drawedge(A128,A127){}
  \drawedge[curvedepth=1](A127,A2){}
  \drawedge[curvedepth=1](A1,A126){}
  \drawedge(A126,A125){}
  \drawedge(A125,A124){}
  \drawedge[curvedepth=1](A124,A2){}
    \drawedge[curvedepth=1](A1,A1218){}
  \drawedge(A1218,A1217){}
  \drawedge(A1217,A1216){}
  \drawedge[curvedepth=1](A1216,A2){}
         \gasset{Nw=7,Nh=7,Nmr=3.5,curvedepth=0}
    \thinlines
    \node[Nmarks=ri, iangle=180](A0)(90,60){$q_0$}
    \node(A1)(115,70){ $q_1$}
   \node(A2)(115,50){$q_2$}
   \node[Nmarks=ri, iangle=180](AA0)(90,20){$q_0$}
    \node(AA1)(115,30){$q_1$}
   \node(AA2)(115,10){$q_2$}
   \node(AA3)(103,24){$q_3$}
   \drawedge(A0,A1){$2z+z^2$}
   \drawedge(A1,A2){$\dfrac{1-\sqrt{1-4z}}{2}$}
   \drawedge(A2,A0){$2z$}
      \drawedge[curvedepth=5](AA0,AA1){$2z$}
      \drawedge(AA0,AA3){}
      \drawedge(AA3,AA1){}
          \drawedge(AA2,AA0){$2z$}
     \drawedge[curvedepth=-3](AA1,AA2){$z$}
     \drawedge[curvedepth=3](AA1,AA2){$\dfrac{z-z\sqrt{1-4z}}{2}$}
     \gasset{Nframe=n,Nadjust=w,Nh=6,Nmr=0}
    \node(P)(-5,74){$\mathcal{M}_1$}
   \node(P)(66,35){$\cdots$}
    \node(P)(80,74){$\mathcal{M}_2$:}
        \node(P)(80,35){$\mathcal{M}_3$:}
         \end{picture}
  \caption{Counting automata with transitions in parallel, Example \ref{ejeconteoo}.}
  \label{ejeautconteo}
\end{figure}
\end{example}

\begin{theorem}[Second Convergence Theorem]\label{teorema2conv}
Let  $\mathcal{M}$ be an automaton, and  let  \linebreak $f_1^q(z), f_2^q(z), \ldots,$ be transitions in parallel from  state $q\in Q$ in a counting automaton of $\mathcal{M}$. Then $\mathcal{M}$ is convergent if the series
\begin{align*}
F^q(z)=\sum_{k=1}^{\infty}f_k^q(z)
\end{align*}
is a convergent series for each visible state $q\in Q$ of the counting automaton.
\end{theorem}

\begin{proposition}\label{sisecuacion}
If $f(z)$ is a polynomial transition in parallel from state $p$  to $q$ in a finite counting automaton  $\mathcal{M}$, then this gives rise to an equation in the  system of GFs equations of $\mathcal{M}$ $$L_p=f(z)L_q + \left[p\in F\right]+ \cdots.$$
\end{proposition}

\begin{proposition}\label{sisecuacion2}
Let $\mathcal{M}$ be a convergent automaton such that  a counting automaton of $\mathcal{M}$ has a finite number  of visible states $q_0,q_1,\ldots,q_r$, in which the number of transitions in parallel  starting from each state is finite.   Let
$f_1^{q_t}(z), f_2^{q_t}(z), \ldots, f_{s(t)}^{q_t}(z)$ be the transitions in parallel from the state $q_t\in Q$. Then the  GF for the  language $L(\mathcal{M})$ is $L_{q_0}(z)$. It is obtained  by solving the system of $r+1$ GFs equations
\begin{multline*}
L({q_t})(z)=f_1^{q_t}(z)L(q_{t_1})(z)+f_2^{q_t}(z)L(q_{t_2})(z)+ \cdots
+f_{s(t)}^{q_t}(z)L(q_{t_{s(t)}})(z)+[q_t\in F],
\end{multline*}
with  $0\leq t\leq r$, where $q_{t_k}$ is the visible state joined with $q_t$ through the transition in parallel $f_k^{q_t}$, and  $L(q_{t_k})$ is the GF for the language accepted by
$\mathcal{M}$ if  $q_{t_k}$ is the initial state.
\end{proposition}

\begin{example}
The system of GFs equations  associated with $\mathcal{M}_2$, see Example \ref{ejeconteoo},  is
\begin{align*}
\begin{cases}
L_0&=(2z+z^2)L_1 + 1 \\
L_1&=\dfrac{1- \sqrt{1-4z}}{2}L_2\\
L_2&=2zL_0.
\end{cases}
\end{align*}
Solving the system for  $L_0$, we find the GF for the language  $\mathcal{M}_2$ and therefore of  $\mathcal{M}_1$ and  $\mathcal{M}_3$
\begin{align*}
L_0=\frac{1}{1-(2z^2+z^3)(1-\sqrt{1-4z})}= 1 + 4z^3 + 6z^4 + 10z^5 + 40z^6 + 114z^7+ \cdots.
\end{align*}
\end{example}

\subsection{An Example of the Counting Automata  Methodology (CAM)}

A counting automaton associated with an automaton $\mathcal{M}$ can be used to model combinatorial objects if there is a bijection between all words recognized by the automaton $\mathcal{M}$  and the combinatorial objects. Such method, along with the previous theorems and propositions constitute the  \textbf{Counting Automata  Methodology (CAM)}, see \cite{ROD}.

We distinguish three phases in the CAM:
\begin{enumerate}
\item Given a problem of enumerative combinatorics, we have to find a convergent automaton $\mathcal{M}$ (see Theorems \ref{teorema1conv} and \ref{teorema2conv}) whose GF is the solution of the problem.

\item Find a general formula for  the GF of $\mathcal{M}'$, where $\mathcal{M}'$ is an automaton obtained from $\mathcal{M}$ truncating a set of states or edges see Propositions \ref{sisecuacion} and \ref{sisecuacion2}. Sometimes we find a relation of iterative type, such as a continued fraction.

\item Find the GF $f(z)$ to which converge the GFs associated to each $\mathcal{M}'$, which is  guaranteed by the Convergences theorems. \end{enumerate}

\begin{example}\label{ejmotzkin}
A \emph{Motzkin path} of length $n$ is a lattice path of  $\mathbb{Z \times Z}$ running from $(0, 0)$ to $(n, 0)$ that never passes below the $x$-axis and whose permitted steps are the up diagonal step $U=(1, 1)$, the down diagonal step $D=(1,-1)$ and the horizontal step $H=(1, 0)$.  The number of Motzkin paths of length $n$ is the \emph{$n$-th Motzkin number} $m_{n}$, sequence  A001006\footnote{Many integer sequences and their properties are found electronically on the On-Line Encyclopedia of Sequences \cite{OEIS}.}.   The number of words of length $n$  recognized by the convergent automaton  $\mathcal{M}_{\mathrm{Mot}}$, see Figure \ref{figmotzkin2},  is  the $n$th Motzkin number and its GF is
$$M\left(z\right)=\sum_{i=0}^{\infty}m_{i}z^{i}=\frac{1-z-\sqrt{1-2z-3z^2}}{2z^2}.$$
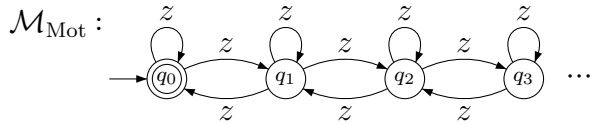
\begin{figure}[h]
  \centering
    \unitlength=3pt
    \begin{picture}(52, 8)(0,-1)
        %\put(0,0){\framebox(52,10){}}
        \gasset{Nw=5,Nh=5,Nmr=2.5,curvedepth=0}
    \thinlines
    \node[Nmarks=ir,iangle=180, curvedepth=3](A0)(0,1){\tiny{$q_{0}$}}
    \node(A1)(15,1){\tiny{$q_{1}$}}
    \node(A2)(30,1){\tiny{$q_{2}$}}
    \node(A3)(45,1){\tiny{$q_{3}$}}
    \drawedge[curvedepth=2.5](A0,A1){$z$}
    \drawedge[curvedepth=2.5](A1,A0){$z$}
    \drawedge[curvedepth=2.5](A1,A2){$z$}
    \drawedge[curvedepth=2.5](A2,A1){$z$}
    \drawedge[curvedepth=2.5](A2,A3){$z$}
    \drawedge[curvedepth=2.5](A3,A2){$z$}
    \drawloop[loopdiam=4,loopangle=90](A0){$z$}
    \drawloop[loopdiam=4,loopangle=90](A1){$z$}
    \drawloop[loopdiam=4,loopangle=90](A2){$z$}
    \drawloop[loopdiam=4,loopangle=90](A3){$z$}
    \gasset{Nframe=n,Nadjust=w,Nh=6,Nmr=0}
    \node(P)(52,1){$\cdots$}
   \node(P)(-14,8){$\mathcal{M}_{\mathrm{Mot}}:$}
    \end{picture}
  \caption{Convergent automaton associated with Motzkin paths.} \label{figmotzkin2}
\end{figure}

In this case the edge from state $q_{i}$ to state  $q_{i+1}$ represents a rise, the edge from the state $q_{i+1}$ to $q_{i}$ represents a fall and the loops represent the level steps, see   Table \ref{tab:mot}.
\begin{table}[h]
  \centering
  \begin{tabular}{|c|c|c|}\hline
  $(q_i, z,q_{i+1})\in E \Leftrightarrow$
\psset{unit=5mm}
\begin{pspicture}(0,0)(2,2)
\psgrid[gridwidth=0.3pt,gridcolor=gray,subgriddiv=0, gridlabels=0]
\psline[linewidth=1pt]{-}(0,0)(2,2)
\end{pspicture}
&
$(q_{i+1}, z, q_{i})\in E \Leftrightarrow$
\psset{unit=5mm}
\begin{pspicture}(0,0)(2,2)
\psgrid[gridwidth=0.3pt,gridcolor=gray,subgriddiv=0, gridlabels=0]
\psline[linewidth=1pt]{-}(0,2)(2,0)
\end{pspicture}
&
$(q_{i}, z,q_{i})\in E \Leftrightarrow$
\psset{unit=5mm}
\begin{pspicture}(0,0)(2,2)
\psgrid[gridwidth=0.3pt,gridcolor=gray,subgriddiv=0, gridlabels=0]
\psline[linewidth=1pt]{-}(0,1)(2,1)
\end{pspicture} \\ \hline
       \end{tabular}
  \caption{Bijection between $\mathcal{M}_{\mathrm{Mot}}$ and Motzkin paths.}
  \label{tab:mot}
\end{table}

Moreover, it is clear that a word is recognized by  $\mathcal{M}_{\mathrm{Mot}}$  if and only if the number of steps to the right and to the left coincide, which ensures that the path is well formed. Then $$m_n=\left|\left\{w\in L(\mathcal{M}_{\mathrm{Mot}}): \left|w\right|=n \right\}\right|=L^{(n)}(\mathcal{M}_{\mathrm{Mot}}).$$

Let $\mathcal{M}_{\mathrm{Mot}s}$, $s\geq 1$ be the automaton  obtained from $\mathcal{M}_{\mathrm{Mot}}$, by deleting the states $q_{s+1}, q_{s+2}, \dots$.  Therefore the system of GFs equations of $\mathcal{M}_{\mathrm{Mot}s}$ is
\begin{align*}
\left\{\begin{aligned}
L_{0}&= zL_{0}+zL_{1}+1,\\
L_{i}&= zL_{i-1} + zL_{i}+ zL_{i+1}, \ \ 1\leq i \leq s-1, \\
L_{s}&= zL_{s-1}+zL_{s}.\end{aligned}
\right.
\end{align*}
Substituting repeatedly into each equation $L_i$, we have
\begin{center}
\begin{tabular}{c}
$L_0=\cfrac{H}{ 1-\cfrac{F^2}{1-\cfrac{F^2}{\cfrac{\vdots }{1-F^2}}}}$
\end{tabular}
\hspace{-3ex}
\raisebox{-2ex}{$\left.\phantom{\begin{matrix} 1 \\ 1 \\ 1  \\ 1 \\ 1
\end{matrix}} \right\}$}
\  \raisebox{-2ex}{$s$ times,}
\end{center}
where $F=\frac{z}{1-z}$ and $H=\frac{1}{1-z}$.  Since $\mathcal{M}_{\mathrm{Mot}}$ is convergent, then  as $s\rightarrow \infty$ we obtain a convergent continued fraction    $M$ of the  GF of $\mathcal{M}_{\mathrm{Mot}}$. Moreover,
\begin{align*}
M=\cfrac{H}{1-F^2\left(\frac{M}{H}\right)}.
\end{align*}
 Hence  $z^2M^2-(1-z)M+1=0$ and
$$M(z)=\frac{1-z\pm \sqrt{1-2z-3z^2}}{2z^2}.$$ Since $\epsilon \in L(\mathcal{M}_{\mathrm{Mot}})$, $M\rightarrow 0$ as $z\rightarrow 0$. Hence, we take the negative sign for the radical in $M(z)$.
\end{example}

\section{Generating Function for the $k$-Fibonacci Paths}

In this section we find the generating function for $k$-Fibonacci paths, grand $k$-Fibonacci paths, prefix $k$-Fibonacci paths and prefix grand $k$-Fibonacci paths, according to the length.

\begin{lemma}[\cite{ROD}]\label{teoflajolet}
The GF of the automaton $\mathcal{M}_{\mathrm{Lin}}$, see Figure \ref{conteolineal}, is
\begin{align*}
E(z)&=\cfrac{1}{1-h_0\left(z\right)-\cfrac{f_0\left(z\right)g_0\left(z\right)}{1-h_1\left(z\right)- \cfrac{f_1\left(z\right)g_1\left(z\right)}{\ddots}}},
\end{align*}
where $f_i(z), g_i(z)$ and $h_i(z)$ are transitions in parallel for all integer $i\geqslant 0$.%, which can be viewed as functions having some radius of convergence around 0.

\begin{figure}[h]
\centering
    \unitlength=3pt
    \begin{picture}(52, 10)(0,-2)
%    \put(0,0){\framebox(52,10){}}
        \gasset{Nw=5,Nh=5,Nmr=2.5,curvedepth=0}
    \thinlines
    \node[Nmarks=ir,iangle=180, curvedepth=3](A0)(0,0){\tiny{0}}
    \node(A1)(15,0){\tiny{1}}
    \node(A2)(30,0){\tiny{2}}
    \node(A3)(45,0){\tiny{3}}
    \drawedge[curvedepth=3](A0,A1){$f_0$}
    \drawedge[curvedepth=3](A1,A0){$g_0$}
    \drawedge[curvedepth=3](A1,A2){$f_1$}
    \drawedge[curvedepth=3](A2,A1){$g_1$}
    \drawedge[curvedepth=3](A2,A3){$f_2$}
    \drawedge[curvedepth=3](A3,A2){$g_2$}
    \drawloop[loopdiam=4,loopangle=90](A0){$h_0$}
    \drawloop[loopdiam=4,loopangle=90](A1){$h_1$}
    \drawloop[loopdiam=4,loopangle=90](A2){$h_2$}
    \drawloop[loopdiam=4,loopangle=90](A3){$h_3$}
    \gasset{Nframe=n,Nadjust=w,Nh=6,Nmr=0}
    \node(P)(52,0){$\cdots$}
    \node(P)(-15,10){$\mathcal{M}_{\mathrm{Lin}}:$}
    \end{picture}
  \caption{Linear infinite counting automaton $\mathcal{M}_{Lin}$}
  \label{conteolineal}
\end{figure}
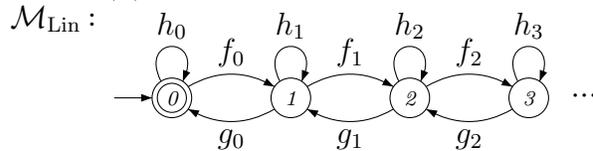
\end{lemma}

The last lemma coincides with Theorem 1 in \cite{FLA} and Theorem 9.1 in \cite{RUT}. However, this presentation extends their applications, taking into account that $f_i(z), g_i(z)$ and $h_i(z)$ are GFs, which can be GFs of  several variables.

\begin{corollary}\label{coro1}
If for all integers $i\geq0$, $f_i(z)=f(z), g_i(z)=g(z)$ and $h_i(z)=h(z)$  in $\mathcal{M}_{\mathrm{Lin}}$, then the GF is
\begin{align}
B(z)&=\frac{1-h(z)-\sqrt{(1-h(z))^2-4f(z)g(z)}}{2f(z)g(z)}  \\
&=\sum_{n=0}^{\infty}\sum_{m=0}^{\infty}C_n \binom{m+2n}{m}\left(f\left(z\right)g\left(z\right)\right)^n\left(h(z)\right)^m \label{coro1ec}\\
&=\cfrac{1}{1-h\left(z\right)-\cfrac{f\left(z\right)g\left(z\right)}{1-h\left(z\right)- \cfrac{f\left(z\right)g\left(z\right)}{1-h\left(z\right) -\cfrac{f\left(z\right)g\left(z\right)}{\ddots}}}},
\end{align}
where $C_n$ is the $n$th Catalan number, sequence A000108.
\end{corollary}

\begin{theorem}\label{TeoFibo1}
The generating function for the $k$-Fibonacci paths  according to the their length is
\begin{align}
T_k(z)&=\sum_{i=0}^{\infty}|\mathcal{M}_{F_{k,i}}|z^i\\
&=\frac{1-(k+1)z-z^2-\sqrt{(1-(k+1)z-z^2)^2-4z^2(1-kz-z^2)^2}}{2z^2(1-kz-z^2)} \label{ecfibo1}\\
&=\cfrac{1}{1-\frac{z}{1-kz-z^2}-\cfrac{z^2}{1-\frac{z}{1-kz-z^2}- \cfrac{z^2}{1-\frac{z}{1-kz-z^2} -\cfrac{z^2}{\ddots}}}}  \label{ecfibo11}
\end{align}
and
\begin{align*}
\left[z^t\right]T_k(z)=\sum_{n=0}^{t}\sum_{m=0}^{t-2n}\binom{m+2n}{m}C_nF_{k,t-2n-m+1}^{(m)},
\end{align*}
where $C_n$ is the $n$-th Catalan number and $F_{k,j}^{(r)}$ is a convolved $k$-Fibonacci number.
\end{theorem}

Convolved $k$-Fibonacci numbers $F_{k,j}^{(r)}$ are defined by $$f_k^{(r)}(x)=(1-kx-x^2)^{-r}=\sum_{j=0}^{\infty}F_{k,j+1}^{(r)}x^j,  \ \ r\in \mathbb{Z}^+.$$
 Note that
\begin{align*}
F_{k, m+1}^{(r)}=\sum_{j_1+j_2+\cdots +j_r=m}F_{k, j_1+1}F_{k, j_2+1}\cdots F_{k, j_r+1}.  %\label{convolfibo}
\end{align*}
Moreover, using a result of Gould \cite[p. 699]{GOU} on Humbert polynomials (with $n = j, m = 2,
x = k/2, y = -1, p = -r$ and $C = 1$), we have
\begin{align*}
F_{k, j+1}^{(r)}=\sum_{l=0}^{\lfloor j/2 \rfloor}\binom{j+r-l-1}{j-l}\binom{j-l}{l}k^{j-2l}.%\label{formula}
\end{align*}
Ram\'irez  \cite{RAM} studied  some properties of convolved $k$-Fibonacci numbers.

\begin{proof}
Equations (\ref{ecfibo1}) and (\ref{ecfibo11}) are clear from Corollary \ref{coro1}  taking  $f(z)=z=g(z)$ and $h(z)=\frac{z}{1-kz-z^2}$. Note that $h(z)$ is the GF of $k$-Fibonacci numbers. In this case the edge from state $q_{i}$ to state  $q_{i+1}$ represents a rise, the edge from the state $q_{i+1}$ to $q_{i}$ represents a fall and the loops represent the  $F_{k,l}-$colored length horizontal steps $(l=1,2,\dots)$. Moreover, from Equation (\ref{coro1ec}), we obtain

\begin{align*}
T_k(z)&=\sum_{n=0}^{\infty}\sum_{m=0}^{\infty}C_n\binom{m+2n}{m}z^{2n}\left(\frac{z}{1-kz-z^2}\right)^m\\
&=\sum_{n=0}^{\infty}\sum_{m=0}^{\infty}C_n\binom{m+2n}{m}z^{2n+m}\left(\frac{1}{1-kz-z^2}\right)^m \\
&=\sum_{n=0}^{\infty}\sum_{m=0}^{\infty}C_n\binom{m+2n}{m}z^{2n+m}\sum_{i=0}^{\infty} F_{k,i+1}^{(m)}z^i\\
&=\sum_{n=0}^{\infty}\sum_{m=0}^{\infty}\sum_{i=0}^{\infty} C_n F_{k,i+1}^{(m)} \binom{m+2n}{m}z^{2n+m+i}, \\
\end{align*}
taking $s=2n+m+i$
\begin{align*}
T_k(z)=\sum_{n=0}^{\infty}\sum_{m=0}^{\infty}\sum_{s=2n+m}^{\infty} C_n F_{k,s-2n-m+1}^{(m)} \binom{m+2n}{m}z^{s}.
\end{align*}
Hence
\begin{align*}
\left[z^t\right]T_k(z)=\sum_{n=0}^{t}\sum_{m=0}^{t-2m} C_n F_{k,t-2n-m+1}^{(m)} \binom{m+2n}{m}.
 \end{align*}
\end{proof}

In Table \ref{tabfibok1} we show the first terms of the sequence $|\mathcal{M}_{F_{k,i}}|$ for $k=1, 2, 3, 4$.

\begin{table}[H]
  \centering
  \begin{tabular}{|c|l|}\hline
$k$  & Sequence    \\ \hline
   1 & 1, 1, 3, 8, 23, 67, 199, 600, 1834, 5674, 17743, \dots  \\  \hline
   2 & 1, 1, 4, 13, 47, 168, 610, 2226, 8185, 30283, 112736, \dots  \\  \hline
   3 & 1, 1, 5, 20, 89, 391, 1735, 7712, 34402, 153898, 690499, \dots   \\  \hline
    4 & 1, 1, 6, 29, 155, 820, 4366, 23262, 124153, 663523, 3551158, \dots  \\  \hline
       \end{tabular}
  \caption{Sequences $|\mathcal{M}_{F_{k,i}}|$ for   $k=1, 2, 3, 4$.}
  \label{tabfibok1}
\end{table}

\begin{definition}\label{fgoei} For all integers $i\geq 0$ we define the continued fraction
 $E_i(z)$ by:
\begin{align*}
E_i(z)&=\cfrac{1}{1-h_i\left(z\right)-\cfrac{f_{i}\left(z\right)g_{i}\left(z\right)}{1-h_{i+1}\left(z\right)- \cfrac{f_{i+1}\left(z\right)g_{i+1}\left(z\right)}{\ddots}}},
\end{align*}
where $f_i(z), g_i(z), h_i(z)$ are transitions in parallel for all  integers positive $i$.

\end{definition}
\begin{lemma}[\cite{ROD}]\label{teoflajoletbi}
The  GF of the automaton $\mathcal{M}_{\mathrm{BLin}}$, see Figure \ref{conteobilineal}, is
\begin{align*}%\label{fgobilineal}
E_b(z)&=\cfrac{1}{1-h_0(z)-f_0(z)g_0(z)E_1(z)-f'_0(z)g'_0(z)E'_1(z)},
\end{align*}
where $f_i(z), f'_i(z), g_i(z), g'_i(z), h_i(z)$ and $h'_i(z)$ are transitions in parallel for all $i\in\mathbb{Z}$.
 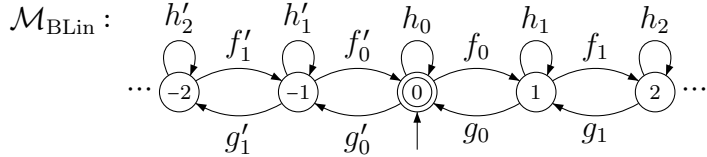
\begin{figure}[h]
 \begin{center}
    \unitlength=3pt
    \begin{picture}(80, 13)(-40,-2)
    %\put(0,0){\framebox(52,5){}}
        \gasset{Nw=5,Nh=5,Nmr=2.5,curvedepth=0}
    \thinlines
    \node[Nmarks=ir,iangle=-90, curvedepth=3](A0)(0,1){\tiny{$0$}}
    \node(A1)(15,1){\tiny{$1$}}
    \node(A2)(30,1){\tiny{$2$}}
    \node(A11)(-15,1){\tiny{$-1$}}
     \node(A22)(-30,1){\tiny{$-2$}}
    \drawedge[curvedepth=3](A0,A1){$f_0$}
    \drawedge[curvedepth=3](A1,A0){$g_0$}
    \drawedge[curvedepth=3](A1,A2){$f_1$}
    \drawedge[curvedepth=3](A2,A1){$g_1$}
    \drawedge[curvedepth=3](A0,A11){$g'_0$}
    \drawedge[curvedepth=3](A11,A22){$g'_1$}
        \drawedge[curvedepth=3](A22,A11){$f'_1$}
    \drawedge[curvedepth=3](A11,A0){$f'_0$}
    \gasset{Nframe=n,Nadjust=w,Nh=6,Nmr=0}
    \node(P)(35,1){$\cdots$}
       \node(P)(-35,1){$\cdots$}
    \node(P)(-45,10){$\mathcal{M}_{\mathrm{BLin}}:$}
       \drawloop[loopdiam=4,loopangle=90](A0){$h_0$}
    \drawloop[loopdiam=4,loopangle=90](A1){$h_1$}
    \drawloop[loopdiam=4,loopangle=90](A2){$h_2$}
    \drawloop[loopdiam=4,loopangle=90](A11){$h'_1$}
       \drawloop[loopdiam=4,loopangle=90](A22){$h'_2$}
    \end{picture}
  \end{center}
  \caption{Linear infinite counting automaton  $\mathcal{M}_{BLin}$.}
  \label{conteobilineal}
\end{figure}
\end{lemma}

\begin{corollary}\label{corobi}
If for all integers  $i$, $f_i(z)=f(z)=f'_i(z), g_i(z)=g(z)=g'_i(z)$ and $h_i(z)=h(z)=h'_i(z)$  in $\mathcal{M}_{\mathrm{BLin}}$, then we have the GF
\begin{align}
B_b(z)&=\frac{1}{\sqrt{(1-h(z))^2-4f(z)g(z)}}\\
&=\cfrac{1}{1-h(z) - \cfrac{2f(z)g(z)}{1-h(z)-\cfrac{f(z)g(z)}{1-h(z)-\cfrac{f(z)g(z)}{\ddots}}}},
\end{align}
where $f(z), g(z)$ and $h(z)$ are  transitions in parallel. Moreover, if  $f(z)=g(z)$, then we have the GF
\begin{align}
B_b(z)=\frac{1}{1-h(z)} + \sum_{n=1}^{\infty}\sum_{k=0}^{\infty}\sum_{l=0}^{\infty}2^n\frac{n}{n+2k}\binom{n+2k}{k}\binom{l+2n+2k}{l}f(z)^{2n+2k}h(z)^{l}.  \label{eccorobi2}
\end{align}

\end{corollary}

\begin{theorem}\label{TeoFibo2}
The generating function for the grand $k$-Fibonacci paths  according to the their length is
\begin{align}
T_k^*(z)&=\sum_{i=0}^{\infty}|\mathcal{M}^*_{F_{k,i}}|z^i=\frac{1-kz-z^2}{\sqrt{(1-(k+1)z-z^2)^2-4z^2(1-kz-z^2)^2}}  \label{ecfibo2}\\
&=\cfrac{1}{1-\frac{z}{1-kz-z^2} - \cfrac{2z^2}{1-\frac{z}{1-kz-z^2}-\cfrac{z^2}{1-\frac{z}{1-kz-z^2}-\cfrac{z^2}{\ddots}}}}  \label{ecfibo22}
\end{align}
and
\begin{align}
\left[z^t\right]T_k^*(z)=F_{k+1,t}^{(1)} + \sum_{n=1}^{t}\sum_{m=0}^{t}\sum_{l=0}^{t-2n-2m}2^n\frac{n}{n+2m}\binom{n+2m}{m}\binom{l+2n+2m}{l}F_{k,t-2n-2m-l+1}^{(l)},  \label{ecfibo3} \end{align}
with $ t\geqslant 1$.
\end{theorem}
\begin{proof}
Equations (\ref{ecfibo2}) and (\ref{ecfibo22}) are clear from Corollary  \ref{corobi}, taking $f(z)=z=g(z)$ and $h(z)=\frac{z}{1-kz-z^2}$. Moreover, from Equation (\ref{eccorobi2}), we obtain

\begin{align*}
T_k^*(z)&=\frac{1}{1-\frac{z}{1-kz-z^2}}+\sum_{n=1}^{\infty}\sum_{m=0}^{\infty}\sum_{l=0}^{\infty}2^n\frac{n}{n+2m}\binom{n+2m}{m}\binom{l+2n+2m}{l}z^{2n+2m}\left(\frac{z}{1-kz-z^2}\right)^l\\
&=1+\sum_{j=0}^{\infty}F_{k+1,j}^{(1)}z^j+\sum_{n=1}^{\infty}\sum_{m=0}^{\infty}\sum_{l=0}^{\infty}\sum_{u=0}^{\infty} 2^n\frac{n}{n+2m}\binom{n+2m}{m}\binom{l+2n+2m}{l}F_{k,u}^{(l)}z^{2n+2m+u+1},
\end{align*}
taking $s=2n+2m+l+u$
\begin{multline*}
T_k^*(z)=1+ \sum_{j=0}^{\infty}F_{k+1,j}^{(1)}z^j +\\ \sum_{n=1}^{\infty}\sum_{m=0}^{\infty}\sum_{l=0}^{\infty}\sum_{s=2n+2m+l}^{\infty} 2^n\frac{n}{n+2m}\binom{n+2m}{m}\binom{l+2n+2m}{l}F_{k,s-2n-2m-l}^{(l)}z^s.
\end{multline*}
Therefore  Equation (\ref{ecfibo3}) is clear.
\end{proof}

In Table \ref{tabfibok2} we show the first terms of the sequence $|\mathcal{M}^*_{F_{k,i}}|$ for $k=1, 2, 3, 4$.

\begin{table}[H]
  \centering
  \begin{tabular}{|c|l|}\hline
$k$  & Sequence    \\ \hline
   1 & 1, 4, 11, 36, 115, 378, 1251, 4182, 14073, 47634, \dots \\  \hline
   2 & 1, 5, 16, 63, 237, 920, 3573, 14005, 55156, 218359,  \dots \\  \hline
   3 & 1, 6, 23, 108, 487, 2248, 10371, 48122, 223977, 1046120, \dots \\  \hline
    4 & 1, 7, 32, 177, 949, 5172, 28173, 153963, 842940, 4624581, \dots \\  \hline
       \end{tabular}
  \caption{Sequences $|\mathcal{M}_{F_{k,i}}^*|$ for   $k=1, 2, 3, 4$ and $i\geqslant1$.}
  \label{tabfibok2}
\end{table}

In Figure \ref{fig2} we show the set $\mathcal{M}_{F_{2,3}}^*$.
   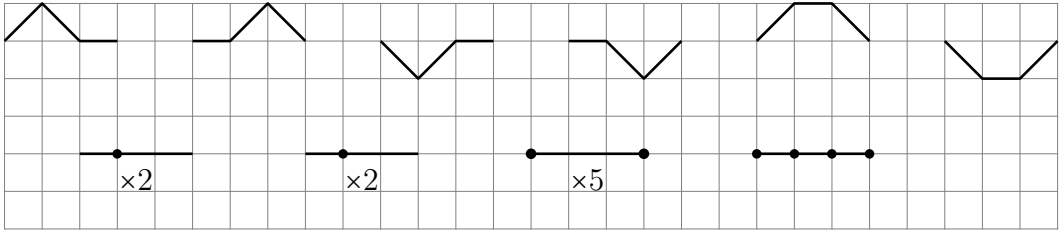
\begin{figure}[H]
\centering
\psset{unit=5mm}
\begin{pspicture}(0,-5)(28,1)
\psgrid[gridwidth=0.3pt,gridcolor=gray,subgriddiv=0, gridlabels=0]
\psline[linewidth=1pt]{-}(0,0)(1,1)(2,0)(3,0)
\psline[linewidth=1pt]{-}(5,0)(6,0)(7,1)(8,0)
\psline[linewidth=1pt]{-}(10,0)(11,-1)(12,0)(13,0)
\psline[linewidth=1pt]{-}(15,0)(16,0)(17,-1)(18,0)
\psline[linewidth=1pt]{-}(20,0)(21,1)(22,1)(23,0)
\psline[linewidth=1pt]{-}(25,0)(26,-1)(27,-1)(28,0)
\psline[linewidth=1pt](20,-3)(23,-3)
\psdots(20,-3)(21,-3)(22,-3)(23,-3)
\psline[linewidth=1pt]{-}(2,-3)(5,-3)
\psdots(3,-3)(9,-3)
\psline[linewidth=1pt]{-}(8,-3)(11,-3)
\psline[linewidth=1pt]{*-*}(14,-3)(17,-3)
\rput(3.5,-3.7){$\times 2$}
\rput(9.5,-3.7){$\times 2$}
\rput(15.5,-3.7){$\times 5$}
\end{pspicture}
\caption{grand $k$-Fibonacci Paths of length 3, $|\mathcal{M}^*_{F_{2,3}}|=16$.}
 \label{fig2}
\end{figure}

\begin{lemma}[\cite{ROD}]\label{autoconteofinal}
The GF of  the automaton \textsc{Fin}$_{\mathbb{N}}(\mathcal{M}_{Lin})$, see Figure \ref{finlineal}, is
\begin{align*}
G(z)=E(z)+\sum_{j=1}^{\infty} \left(\prod_{i=0}^{j-1}(f_i(z)E_i(z))E_j(z) \right),  \label{ecfinn}
\end{align*}
where $E(z)$ is the GF in Lemma  \ref{teoflajolet}.
 \begin{figure}[h]
  \centering
    \unitlength=3pt
    \begin{picture}(52, 14)(-5,-2)
%    \put(0,0){\framebox(52,10){}}
        \gasset{Nw=5,Nh=5,Nmr=2.5,curvedepth=0}
    \thinlines
    \node[Nmarks=ir,iangle=180, curvedepth=3](A0)(0,0){\tiny{0}}
    \node[Nmarks=r](A1)(15,0){\tiny{1}}
    \node[Nmarks=r](A2)(30,0){\tiny{2}}
    \node[Nmarks=r](A3)(45,0){\tiny{3}}
    \drawedge[curvedepth=2.5](A0,A1){$f_0$}
    \drawedge[curvedepth=2.5](A1,A0){$g_0$}
    \drawedge[curvedepth=2.5](A1,A2){$f_1$}
    \drawedge[curvedepth=2.5](A2,A1){$g_1$}
    \drawedge[curvedepth=2.5](A2,A3){$f_2$}
    \drawedge[curvedepth=2.5](A3,A2){$g_2$}
    \drawloop[loopdiam=4,loopangle=90](A0){$h_0$}
    \drawloop[loopdiam=4,loopangle=90](A1){$h_1$}
    \drawloop[loopdiam=4,loopangle=90](A2){$h_2$}
    \drawloop[loopdiam=4,loopangle=90](A3){$h_3$}
    \gasset{Nframe=n,Nadjust=w,Nh=6,Nmr=0}
    \node(P)(52,0){$\cdots$}
    \node(P)(-17,8){\textsc{Fin}$_{\mathbb{N}}(\mathcal{M}_{Lin}):$}
    \end{picture}
  \caption{Linear  infinite counting automaton  \textsc{Fin}$_{\mathbb{N}}(\mathcal{M}_{Lin})$.}
  \label{finlineal}
\end{figure}
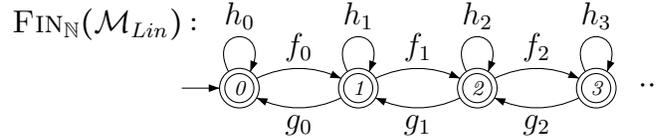
\end{lemma}

\begin{corollary}\label{corofint}
If for all integer  $i\geqslant 0$, $f_i(z)=f(z), g_i(z)=g(z)$ and $h_i(z)=h(z)$ in \textsc{Fin}$_{\mathbb{N}}(\mathcal{M}_{Lin})$, then the GF is:
\begin{align}
G(z)&=\frac{1-2f(z)-h(z)-\sqrt{(1-h(z))^2-4f(z)g(z)}}{2f(z)\left(f(z)+g(z)+h(z)-1\right)}\\
&=\cfrac{1}{1-f(z)-h(z) - \cfrac{f(z)g(z)}{1-h(z)-\cfrac{f(z)g(z)}{1-h(z)-\cfrac{f(z)g(z)}{\ddots}}}},  \label{ecfin}
\end{align}
where $f(z), g(z)$ and $h(z)$ are transitions in parallel and $B(z)$ is the GF in Corollary  \ref{coro1}. Moreover, if   $f(z)=g(z)$  and $h(z)\neq0$, then we obtain the  GF
\begin{align}
G(z)&=\sum_{n=0}^{\infty}\sum_{k=0}^{\infty}\sum_{l=0}^{\infty}\frac{n+1}{n+k+1}\binom{n+2k+l}{k, l, k+n}f^{2k+n}(z)h^l(z).
\end{align}
\end{corollary}

\begin{theorem}
The generating function for the prefix $k$-Fibonacci paths  according to the their length is
\begin{align*}
PT_k(z)&=\sum_{i=0}^{\infty}|\mathcal{PM}_{F_{k,i}}|z^i\\
&=\frac{(1-2z)(1-kz-z^2)-z-\sqrt{(1-z(k+1)-z^2)^2+4z^2(1-kz-z^2)^2}}{2z((1-kz-z^2)(2z-1)+z)}
\end{align*}
and
\begin{align*}
\left[z^t\right]PT_k(z)=\sum_{n=0}^{t}\sum_{m=0}^{t}\sum_{l=0}^{t-2m-n}\frac{n+1}{n+m+1}\binom{n+2m+l}{m,l,m+n}F_{k,t-2m-n-l+1}^{(l)},  \ t\geqslant 0.
\end{align*}
\end{theorem}
\begin{proof}
The proof is analogous to the proof of Theorem \ref{TeoFibo1} and \ref{TeoFibo2}.
\end{proof}

In Table \ref{tabfibok3} we show the first terms of the sequence $|\mathcal{PM}_{F_{k,i}}|$ for $k=1, 2, 3, 4$.

\begin{table}[h]
  \centering
  \begin{tabular}{|c|l|}\hline
$k$  & Sequence    \\ \hline
   1 & 1, 2, 6, 19, 62, 205, 684, 2298, 7764, 26355, 89820, \dots \\  \hline
   2 & 1, 2, 7, 26, 101, 396, 1564, 6203, 24693, 98605, 394853,  \dots \\  \hline
   3 & 1, 2, 8, 35, 162, 757, 3558, 16766, 79176, 374579, 1775082, \dots \\  \hline
    4 & 1, 2, 9, 46, 251, 1384, 7668, 42555, 236463, 1315281, 7322967,  \dots \\  \hline
       \end{tabular}
  \caption{Sequences $|\mathcal{PM}_{F_{k,i}}|$ for   $k=1, 2, 3, 4$.}
  \label{tabfibok3}
\end{table}

In Figure \ref{fig3} we show the set $\mathcal{MP}_{F_{2,3}}$.
   \begin{figure}[H]
\centering
\psset{unit=5mm}
\begin{pspicture}(0,-9)(28,2)
\psgrid[gridwidth=0.3pt,gridcolor=gray,subgriddiv=0, gridlabels=0]
\psline[linewidth=1pt]{-}(0,0)(1,1)(2,0)(3,0)
\psline[linewidth=1pt]{-}(5,0)(6,0)(7,1)(8,0)
\psline[linewidth=1pt]{-}(10,0)(11,1)(12,2)(13,2)
\psline[linewidth=1pt]{-}(15,0)(16,1)(17,2)(18,1)
\psline[linewidth=1pt]{-}(20,0)(21,1)(22,1)(23,0)
\psline[linewidth=1pt]{-}(25,0)(26,1)(27,1)(28,2)
\psline[linewidth=1pt]{-}(0,-3)(3,-3)
\psdots(1,-3)(7,-3)
\psline[linewidth=1pt]{-}(5,-3)(8,-3)
\psline[linewidth=1pt]{*-*}(10,-3)(13,-3)
\psline[linewidth=1pt](15,-3)(18,-3)
\psline[linewidth=1pt](20,-3)(21,-2)(23,-2)
\psline[linewidth=1pt](25,-3)(26,-2)(28,-2)
\psdots(15,-3)(16,-3)(17,-3)(18,-3)(21,-2)(22,-2)(23,-2)(26,-2)(28,-2)
\psline[linewidth=1pt]{-}(0,-8)(3,-5)
\psline[linewidth=1pt](5,-8)(6,-7)(7,-8)(8,-7)
\psline[linewidth=1pt](10,-8)(11,-8)(13,-6)
\psline[linewidth=1pt](15,-8)(16,-8)(17,-7)(18,-7)
\psline[linewidth=1pt](20,-8)(22,-8)(23,-7)
\psline[linewidth=1pt](25,-8)(27,-8)(28,-7)
\psdots(20,-8)(21,-8)(22,-8)(25,-8)(27,-8)
\rput(1.5,-3.7){$\times 2$}
\rput(6.5,-3.7){$\times 2$}
\rput(11.5,-3.7){$\times 5$}
\rput(26.5,-3.7){$\times 2$}
\rput(26.5,-8.7){$\times 2$}
\end{pspicture}
\caption{prefix $k$-Fibonacci paths of length 3, $|\mathcal{PM}_{F_{2,3}}|=26$.}
 \label{fig3}
\end{figure}
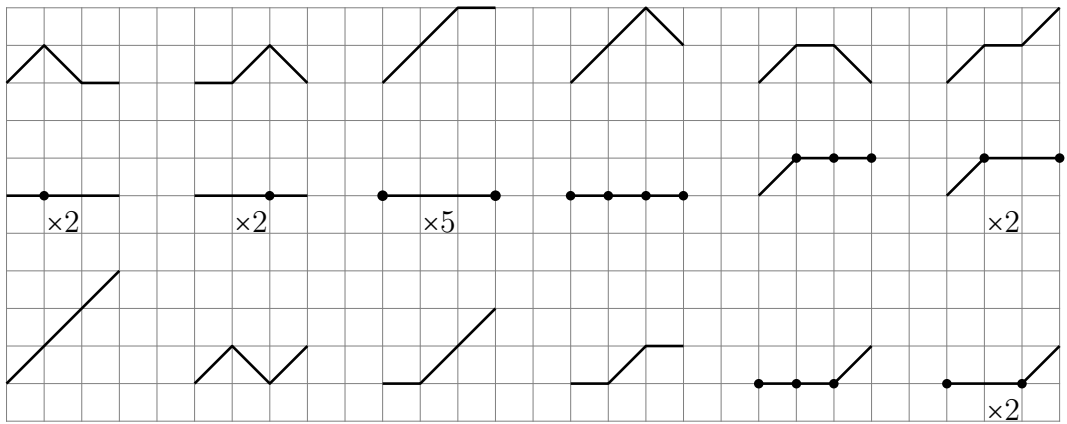

 \begin{lemma}\label{autoconteobifinal}
The GF of the automaton  \textsc{Fin}$_{\mathbb{Z}}(\mathcal{M}_{BLin})$, see Figure \ref{finbilineal}, is
\begin{align*}
H(z)&=\frac{EE'}{E + E'-EE'(1-h_0)}\left(1+ \sum_{j=1}^{\infty}\prod_{k=1}^{j-1}f_kE_k f_0E_j +  \sum_{j=1}^{\infty}\prod_{k=1}^{j-1}g'_kE'_k g'_0E'_j  \right)\\
&=\frac{E'(z)G(z)+E(z)G'(z)-E(z)E'(z)}{E(z)+E'(z)-E(z)E'(z)(1-h_0(z))},
\end{align*}
\normalsize
where $G(z)$ is the  GF in Lemma  \ref{autoconteofinal} and $G'(z), E'(z)$ are the GFs obtained from  $G(z)$ and $E(z)$ changing  $f(z)$ to $g'(z)$ and $g(z)$ to  $f'(z)$.
 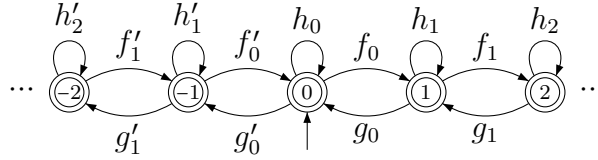
\begin{figure}[h]
 \begin{center}
    \unitlength=3pt
    \begin{picture}(80, 12)(-40,-2)
    %\put(0,0){\framebox(52,5){}}
        \gasset{Nw=5,Nh=5,Nmr=2.5,curvedepth=0}
    \thinlines
    \node[Nmarks=ir,iangle=-90, curvedepth=3](A0)(0,1){\tiny{$0$}}
    \node[Nmarks=r](A1)(15,1){\tiny{$1$}}
    \node[Nmarks=r](A2)(30,1){\tiny{$2$}}
    \node[Nmarks=r](A11)(-15,1){\tiny{$-1$}}
     \node[Nmarks=r](A22)(-30,1){\tiny{$-2$}}
    \drawedge[curvedepth=3](A0,A1){$f_0$}
    \drawedge[curvedepth=3](A1,A0){$g_0$}
    \drawedge[curvedepth=3](A1,A2){$f_1$}
    \drawedge[curvedepth=3](A2,A1){$g_1$}
    \drawedge[curvedepth=3](A0,A11){$g'_0$}
    \drawedge[curvedepth=3](A11,A22){$g'_1$}
        \drawedge[curvedepth=3](A22,A11){$f'_1$}
    \drawedge[curvedepth=3](A11,A0){$f'_0$}
    \gasset{Nframe=n,Nadjust=w,Nh=6,Nmr=0}
    \node(P)(36,1){$\cdots$}
       \node(P)(-36,1){$\cdots$}
       \drawloop[loopdiam=4,loopangle=90](A0){$h_0$}
    \drawloop[loopdiam=4,loopangle=90](A1){$h_1$}
    \drawloop[loopdiam=4,loopangle=90](A2){$h_2$}
    \drawloop[loopdiam=4,loopangle=90](A11){$h'_1$}
       \drawloop[loopdiam=4,loopangle=90](A22){$h'_2$}
    \end{picture}
  \end{center}
  \caption{Linear infinite counting automaton \textsc{Fin}$_{\mathbb{Z}}(\mathcal{M}_{BLin})$.}
  \label{finbilineal}
\end{figure}

Moreover, if for all integer $i\geqslant 0$, $f_i(z)=f(z)=f'_i(z), g_i(z)=g(z)=g'_i(z)$ and $h_i(z)=h(z)=h'_i(z)$ in \textsc{Fin}$_{\mathbb{Z}}(\mathcal{M}_{BLin})$, then the GF is
\begin{align}
H(z)&=\frac{1}{1-f(z)-g(z)-h(z)}.
\end{align}
\end{lemma}

\begin{theorem}
The generating function for the prefix grand $k$-Fibonacci paths according to the their length is
\begin{align*}
PT^*_k(z)=\sum_{i=0}^{\infty}|\mathcal{PM}_{F^*_{k,i}}|z^i=\frac{1-kz-z^2}{1-(k+3)z-(1-2k)z^2+2z^3}.
\end{align*}
\end{theorem}
\begin{proof}
The proof is analogous to the proof of Theorem  \ref{TeoFibo1} and \ref{TeoFibo2}.

\end{proof}

In Table \ref{tabfibok4} we show the first terms of the sequence $|\mathcal{PM}^*_{F_{k,i}}|$ for $k=1, 2, 3, 4$.

\begin{table}[h]
  \centering
  \begin{tabular}{|c|l|}\hline
$k$  & Sequence   \\ \hline
   1 & 1, 3, 10, 35, 124, 441, 1570, 5591, 19912, 70917, 252574,  \dots \\  \hline
   2 & 1, 3, 11, 44, 181, 751, 3124, 13005, 54151, 225492, 938997,  \dots \\  \hline
   3 & 1, 3, 12, 55, 264, 1285, 6280, 30727, 150392, 736157, 3603528, \dots  \\  \hline
    4 & 1, 3, 13, 68, 379, 2151, 12268, 70061, 400249, 2286780, 13065595  \dots \\  \hline
       \end{tabular}
  \caption{Sequences $|\mathcal{PM}^*_{F_{k,i}}|$ for   $k=1, 2, 3, 4$.}
  \label{tabfibok4}
\end{table}

\section{Acknowledgments}
The second author was partially supported by Universidad Sergio Arboleda under Grant no. DII-
262.

\end{document}